\documentclass[10pt,leqno]{article}
\usepackage[margin=3.cm]{geometry}
\usepackage{cite}
\usepackage{graphicx}
\usepackage{amsthm}
\usepackage{amsmath}
\usepackage{amsfonts}
\usepackage{amssymb}
\usepackage{enumerate} 
\usepackage{mathrsfs} 
\numberwithin{equation}{section}
\newcommand{\f}{\frac}
\newcommand{\p}{\partial}

\newcommand{\Z}{\mathbb{Z}}
\newcommand{\N}{\mathbb{N}}

\newcommand{\R}{\mathbb{R}}

\newcommand{\E}{\mathbb{E}}

\let\P\BP

\newcommand{\eps}{\varepsilon}

\newcommand{\pnorm}[2]{\left\|#1\right\|_{#2}}

\newcommand{\floor}[1]{\lfloor #1 \rfloor}

\newcommand{\ind}{{\mathbf 1}}

\newcommand{\limity}[1]{\lim_{#1 \to \infty}}

\newcommand\un{\underline}

\newcommand{\calC}{\mathcal{C}}

\newcommand{\calF}{\mathcal{F}}

\newcommand{\calG}{\mathcal{G}}

\newcommand{\calN}{\mathcal{N}}

\newcommand{\calR}{\mathcal{R}}

\theoremstyle{definition}
\newtheorem{thm}{Theorem}[section]
\newtheorem{cor}[thm]{Corollary}
\newtheorem{lemma}[thm]{Lemma}
\newtheorem{prop}[thm]{Proposition}
\newtheorem{conjecture}[thm]{Conjecture}

\newtheorem{remark}[thm]{Remark}

\theoremstyle{definition}

\newcommand{\email}[1]{{\it  E-mail address:}\  \textsf{#1}\\*\protect}
\newcommand{\address}[1]{\small \textsc{#1}\\*\protect}

\author{Spencer Frei $\mbox{}^{1}$ \hspace{1cm} Edwin Perkins $\mbox{}^{2}$ }
\title{A lower bound for $p_c$ in range-$R$ bond percolation in two and three dimensions}

\begin{document}
\maketitle
\begin{center}
\address{Department of Statistics, UCLA}
\address{8125 Math Sciences Building, Los Angeles, CA 90095, USA}
\email{spencerfrei@ucla.edu}
\mbox{}\\
\address{ Department of Mathematics, The University of British Columbia,}
\address{ 1984 Mathematics Road, Vancouver, B.C., Canada V6T 1Z2}
\email{perkins@math.ubc.ca}
\end{center}

\begin{abstract}
We use the connection between bond percolation and SIR epidemics to establish lower bounds for the critical percolation probability in $2$ and $3$ dimensions as the range becomes large.  The bound agrees with the conjectured asymptotics for the long range critical probability, refines results of M. Penrose, and complements results of van der Hofstad and Sakai in dimensions greater than $6$. 
\end{abstract}
\section{Introduction}
\setcounter{equation}{0}
\subsection{Range-$R$ bond percolation}
We study the critical probability in range-$R$ bond percolation.   For a parameter $R\in \N$ called the \textit{range}, define $\Z^d/R = \{ x/R : x\in \Z^d\}$.  We construct an undirected graph $\Z^d_R$ with vertex set $\Z^d/R$ and assign edges between two vertices $x,y\in \Z^d/R$ if $0<\pnorm{x-y}\infty \leq 1$, where $\pnorm \cdot \infty$ denotes the $\ell^\infty$ norm on $\R^d$.  Write $x\sim y$ if there exists an edge between $x$ and $y$ in $\Z^d/R$, let $\calN(x)$ denote the 
set of neighbours of $x$, and denote its size by
\[ V(R) := |\calN(x)|=|\{ y\in \Z^d/R, y\sim 0 : 0 < \pnorm y \infty \leq 1 \}| = (2R+1)^d-1 \ , \]
where $|S|$ denotes the cardinality of a finite set $S$.
Let $E(\Z^d_R)$ denote the set of edges in $\Z^d_R$.
The structure of this graph is unchanged if one scales the lattice so that the vertex set is $\Z^d$ and there are edges between points $x,y\in \Z^d$ when $0<\pnorm{x-y} \infty\leq R$, but for the remainder of this paper we shall focus solely on the graph $\Z^d_R$ for $d\le 3$.  If $x\sim y$, we let $(x,y)$ or $(y,x)$ denote the edge between $x$ and $y$.

\vfill
{\footnoterule
\footnotesize
\noindent 
\today\\
AMS 2000 {\it subject classifications}. Primary 60K35. 
Secondary 60J68, 60J80, 92D30. 
 \\
{\it Keywords and phrases}. Long range bond percolation,  critical probability, SIR epidemic.\\
 \\
{\it Running head}. A lower bound for $p_c$\\
1. Supported by an NSERC Discovery grant. \\ 
2. Supported by an NSERC Discovery grant.\\}
\pagebreak

\vskip 0.3in 

We construct a random, undirected subgraph $G=G_R$ with vertex set $\Z^d_R$ by considering a collection of i.i.d. Bernoulli random variables $\{B(e):e\in E(\Z^d_R)\}$, each with parameter $p >0$.   An edge $e$ is \textit{open} iff $B(e)=1$, while edges with $B(e)=0$ are  \textit{closed}. $G$ is the resulting subgraph with edge set equalling the set of open edges.   Two vertices $x$ and $y$ in $\Z^d_R$  are \textit{connected} if there is a path between $x$ and $y$ consisting of open edges; we denote this event by $x\leftrightarrow y$. Let $\lambda =V(R)p> 0$ denote the mean number of neighbours in $G$ of any vertex.
The cluster $\calC_x$ in $G$ containing $x$ is
\[ \calC_x := \{ y\in \Z^d/R : x \leftrightarrow y \}, \]
and the \textit{percolation probability}, $q(p)$, is denoted by
\[q(p) = \P_p( |\calC_0| = \infty ). \]

The \textit{critical probability} is defined by
\[ p_c =p_c(R):= \inf \{ p : q(p) > 0 \} \ , \]
and the associated critical value of $\lambda$ is 
\[ \lambda_c = \lambda_c(R):= p_c(R)V(R) \ . \]
The obvious monotonicity in $p$ shows that $q(p)=0$ for $p\in[0,p_c)$ and $q(p)>0$ for $p\in(p_c,1]$. 
M. Penrose~\cite{penrose} showed that
\begin{equation}\limity R \lambda_c(R) = 1 \ . \label{pcconv}
\end{equation}
If we view the set of vertices distance $n$ in the $G$-graph distance from the origin as 
a set-valued ``interactive branching process" $\eta_n$, then for $R$ large and $\lambda$ bounded, $\eta_n$ should be well-approximated by an ordinary Galton-Watson branching process with offspring mean $\lambda$.  This is because with so many potential percolation steps from generation $n$ to $n+1$, it is unlikely that the process will take a step to 
a previously visited site.  As a result one expects the critical mean $\lambda_c(R)$ to be close to $1$, the critical mean for the GW branching process. (This intuition is of course well-known and was pointed by Penrose, among others.) Write $a(R)\sim b(R)$ iff the ratio approaches $1$ as $R\to\infty$.  Van der Hofstad and Sakai \cite{remco}
have obtained finer asymptotics on $\lambda_c(R)$ for $d>6$ using the lace expansion:
\begin{align}\label{highd}
&\lambda_c(R)-1\sim \frac{\theta_d}{R^d}
\end{align}
where $\theta_d$ has an explicit expression in terms of a random walk with uniform steps on $[-1,1]^d$.
The extension of \eqref{highd} to $d\ge 4$ (with logarithmic corrections in $d=4$) is the subject of ongoing work of one of us [EP] with Xinghua Zheng.  In the next subsection we recall a parallel conjecture (Conjecture~\ref{pcR}) for $d=2,3$ suggested in Lalley-Perkins-Zheng~\cite{lalleyperkinszheng}. Our main result (Theorem~\ref{thm:lambdac.lowerbound} below) is a lower bound on $\lambda_c(R)$ which confirms the conjecture aside from the constant.  
The first step is to make the above connection between bond percolation and ``interactive branching processes" more precise in the next subsection.

\subsection{SIR epidemic models and bond percolation}
\label{sec:intro.connection.sir}
We recall a well-known connection between bond percolation and an epidemic model which dates back at least to \cite{mollison}.
For the SIR epidemic model on $\Z^d_R$, each vertex
$x\in\Z^d/R$ can be in one of three possible states: we denote $x\in \xi_n$ if $x$ is susceptible at time $n$; $x\in \eta_n$ if $x$ is infected at time $n$; and $x\in \rho_n$ if $x$ is recovered at time $n$.  Each vertex $x$ is in exactly one of the three states, so that for each time $n$, 
\[ \Z^d/R = \xi_n \dot\cup \eta_n \dot\cup \rho_n \ . \]
The epidemic starts from finite initial configurations of infected sites, $\eta_0$, and recovered sites, $\rho_0$.   
An infected site $x\in \eta_n$ infects a susceptible neighbour $y\in \xi_n,\ y\sim x,$ with probability $p=p(R)$, where the infection events are conditionally independent given the current configuration of states.  Infected sites are infected for unit time, after which they become recovered, and recovered sites are immune from infection, so that if $x\in \eta_n$, then $x\in \rho_k$ for all $k \geq n+1$.  Note that
\begin{equation}\rho_n=\rho_0\dot\cup\eta_0\dot\cup\dots\dot\cup\eta_{n-1}, \text{ and is finite for all }n.\label{rangedef}
\end{equation}

We let
\begin{equation}
\p C_n:=\{e=(x,y)\in E(\Z^d_R):x\in\eta_n,y\in\xi_n\}
\label{eq:boundaryCndef}
\end{equation} be the set of possible infection edges leading to $\eta_{n+1}$, where in describing edges $(x,y)$ in $\p C_n$ we will use the convention that the first coordinate refers to the site in $\eta_n$. So that for any given edges $(x_i,y_i),i\le m$, the event $\p C_n=\{(x_i,y_i):i\le m\}$ will include this specification. (This amounts to choosing a particular representative for an equivalence class in our notation.) Let $\calF_n = \sigma(\rho_0,\eta_k, k \leq n)$. Clearly the finite random set $\p C_n$ is $\sigma(\eta_n,\rho_n)\subset\calF_n$-measurable because $\xi_n=(\eta_n\cup\rho_n)^c$. 
If $y$ is a site in $\Z_R^d$, let 
\begin{equation}
D_n(y)=\{x\in\eta_n:(x,y)\in\p C_n\}.
\label{eq:Dndef}
\end{equation}
If 
$S=\{(x_i,y_i):i\le m\}$ is a set of distinct   edges in $\Z^d_R$ and $V\subset \{y_i:i\le m\}:=V_2$, then the above description implies that 
\begin{equation}\label{etalaw}
\P(\eta_{n+1}=V|\calF_n)= \prod_{y\in V_2\setminus V}(1-p)^{|D_n(y)|}\prod_{y\in V}[1-(1-p)^{|D_n(y)|}]\text{ a.s. on }\{\p C_n=S\}.
\end{equation}
Clearly \eqref{etalaw} and the joint law of $(\eta_0,\rho_0)$ uniquely determines
the law of the SIR epidemic $\eta$. 
Since $\rho_{n+1}=\rho_n\cup\eta_n$, \eqref{etalaw} also gives the (time-homogeneous)   Markov property of $(\eta_n, \rho_n)$:  
\begin{equation} \P( (\eta_{n+1},\rho_{n+1}) \in \cdot\ \big| \calF_n) = \P( (\eta_{n+1},\rho_{n+1}) \in \cdot\ \big| \eta_n, \rho_n)=\P_{\eta_n,\rho_n}((\eta_1,\rho_1)\in \cdot)\ , \label{markovprop}\end{equation}
where $\P_{\eta_0,\rho_0}$ is the law of $(\eta,\rho)$ starting at $(\eta_0,\rho_0)$, whenever $(\eta_0,\rho_0)$ are disjoint finite sets of sites.
Unless otherwise indicated we will assume the SIR epidemic under $\P$ starts from $\eta_0 = \{0\}$ and $\rho_0 = \emptyset$, and assume $p = p(R)$, is such that
\[ \lambda(R) := p(R) V(R) \ge 1. \]
We may use the edge percolation variables $\{B(e):e\in E(\Z^d_R)\}$ from the last section to define the infection dynamics of $\eta$ as follows: 
Every infected--susceptible pair $(x,y)\in\p C_n$ has a successful infection iff $B(x,y)=1$.    {\it The dynamics of the SIR epidemic force each bond variable $B(x,y)$ to be used at most once in defining the epidemic process, precisely when $x$ is infected and $y$ is susceptible or conversely.} The fact that the above specification of the infection dynamics leads to \eqref{etalaw} and hence defines an SIR epidemic, is immediate from Lemma~\ref{lem:unexplored-independent} below.

We collect the dynamics here:
\begin{align}
\nonumber \eta_{n+1} &= \bigcup_{x\in \eta_n} \{ y\in \xi_n : B(x,y)=1 \} , \\
\rho_{n+1} &= \rho_n \cup \eta_n ,\label{sirperc} \\
\nonumber\xi_{n+1} &= \xi_n \setminus \eta_{n}. 
\end{align}
Let $d_G(x,y)$ denotes the graph distance in the percolation graph, $G$, between $x,y\in\Z^d_R$ (it may be infinite), and similarly let $d_G(A,x)$ denote the distance in $G$ between a set of vertices $A$ and $x\in\Z^d_R$.  An easy induction (see below) shows that if $(\eta_0,\rho_0)=(\{0\},\emptyset)$, then for  $n\in\Z_+$, $\eta_n=\{x\in\Z^d_R:\,d_G(0,x)=n\}$.  More generally for a given pair of disjoint finite sets of sites in $\Z^d_R$, $(\eta_0,\rho_0)$, let $G(\rho_0)$ be the percolation graph but where all edges containing a vertex in $\rho_0$ are closed.  Then for $(\eta,\rho)$ starting at $(\eta_0,\rho_0)$, 
\begin{equation}\label{graphdistance}\eta_n=\{x\in\Z_R^d:d_{G(\rho_0)}(\eta_0,x)=n\}:=\eta_n^{\eta_0,\rho_0}.
\end{equation}
We present the inductive argument for the above. If $n=0$, this is obvious.  Assume the result for $n$.  Assume first $x\in\eta_{n+1}$. Then by \eqref{sirperc}, $x\in\xi_n\subset\rho_0^c$ and there is an $x'\in\eta_n$ such that $B(x',x)=1$.  By hypothesis, $d_{G(\rho_0)}(\eta_0,x')=n$.  Since $x'\notin\rho_0$ (or else the above distance would be $\infty$) this implies that 
$d_{G(\rho_0)}(\eta_0,x)\le n+1$.  As $x\notin \cup_{k\le n}\eta_k=\rho_{n}\cup\eta_n$ ($x\in\xi_n$ is disjoint from this union by \eqref{sirperc}), the induction hypothesis imples $d_{G(\rho_0)}(\eta_0,x)> n$ and so $d_{G(\rho_0)}(\eta_0,x)=n+1$.  Conversely assume that $d_{G(\rho_0)}(\eta_0,x)=n+1$. This again implies $x\not\in\rho_0$. There is an $x'\notin\rho_0$ so that $d_{G(\rho_0)}(\eta_0,x')=n$, $x'\sim x$ and $B(x',x)=1$. By hypothesis $x'\in\eta_n$. Also by the induction hypothesis $x\not\in\cup_{k=1}^n\eta_k=\rho_{n}\cup\eta_n$ (or else $d_{G(\rho_0)}(\eta_0,x)\le n$). This means $x$ must be in $\xi_n$ and so by \eqref{sirperc} we conclude that $x\in\eta_{n+1}$.  This completes the induction.

Here is one simple consequence of the coupling of initial conditions that comes from this construction. It will be used in Section~\ref{extinctionpf}.
\begin{lemma}
\label{lem:wipe.away.history} For any disjoint finite sets $\eta_0,\rho_0\subset \Z^d_R$ and any $n\in\Z_+$,\\
(a) $\P_{\eta_0,\rho_0}(|\cup_{i=0}^n\eta_i|\le r)\ge \P_{\eta_0,\emptyset}(|\cup_{i=0}^n\eta_i|\le r)\quad\forall r\ge0$.

\noindent(b) $\E_{\eta_0,\emptyset}(|\rho_n|)\le |\eta_0|E_{\{0\},\emptyset}(|\rho_n|)$.
\end{lemma}
\begin{proof} (a) Clearly $d_{G(\rho_0)}(\eta_0,x)\le n$ implies $d_G(\eta_0,x)\le n$, since the former gives the existence of a chain of at most $n$ open bonds in $G$ starting at a vertex in $\eta_0$ and ending at $x$ with the additional property that each vertex is not in $\rho_0$, and this clearly implies the latter.  The result is now immediate from the characterization of $\eta_n$ given in \eqref{graphdistance}. 
\medskip

\noindent (b) As we are dealing with coupled versions of $\rho_n$ all starting at $\rho_0=\emptyset$ but with different $\eta_0$'s, we let $\rho_n^{\eta_0}=\cup_{k=0}^{n-1}\eta_k^{\eta_0,\emptyset}$.  
As $\rho_0=\emptyset$ now, by \eqref{rangedef} and \eqref{graphdistance},
\[\rho_n^{\eta_0}=\{y\in\Z_R^d:d_G(\eta_0,y)< n\}=\cup_{x\in\eta_0}\{y\in\Z_R^d:d_G(x,y)< n\},\]
the last by an elementary argument.  This shows that $|\rho_n^{\eta_0}|\le \sum_{x\in\eta_0}|\rho_n^{\{x\}}|$, and so taking means we may conclude that
\[\E_{\eta_0,\emptyset}(|\rho_n|)\le\sum_{x\in\eta_0}\E_{\{x\},\emptyset}(|\rho_n|) = |\eta_0|\E_{\{0\},\emptyset}(|\rho_n|).\]
The last equality follows from translation invariance of the dynamics of $\eta$. 
\end{proof}

 We say that an epidemic \textit{survives} if with positive probability, for all $n\in \N$, we have $\eta_n\neq \emptyset$.  If with probability one, for some finite $n$ we have $\eta_n=\emptyset$, we say the epidemic becomes \textit{extinct}.  

If we have survival of the epidemic, then with positive probability there is an infinite sequence of sites $x_{k}$ such that $x_{k} \in \eta_{k}$, $x_{k} \sim x_{k-1}$, and $x_{k-1}$ infected $x_{k}$ at time $k$, and hence the edge $B(x_{k-1}, x_{k})$ is open.  This implies that with positive probability, we have percolation from $\eta_0 = \{0\}$ to infinity in range-$R$ bond percolation.  Likewise, percolation from $\eta_0$ to infinity in the percolation model induces an infinite sequence of infections and hence survival in the epidemic.  In this way, percolation to infinity is equivalent to survival of the analogous SIR epidemic on the same graph.  

To understand the large $R$ behaviour of $p_c(R)$ we will fix a parameter $\theta> 0$, rescale our model by $R^{\gamma}$, choose $p=p(R)$ so that $\lambda:=V(R)p(R)=1+\frac{\theta}{R^\gamma}$, and work with the empirical process of the rescaled epidemic model,
\[X_t^R= \f 1{R^\gamma} \sum_{x\in \eta_{\floor{tR^\gamma}}} \delta_{x/(R\sqrt {R^\gamma})}. \]
We first have resealed the range to size $1$ and then applied the usual Brownian scaling due to the time rescaling by $R^\gamma t$.
We will assume $ \theta\le 1$ which implies that $p\in(0,1)$.  In the above $\gamma>0$ must be chosen carefully if we want to obtain an interesting limit. For example, if $\gamma$ is too small, then the range $R=(R^\gamma)^{1/ \gamma}$ will be too large relative to the  scaling parameter $R^\gamma$ and the effect of suppressing infections onto recovered sites will become negligible. As a result we will recover the scaling limit of branching random walk, super-Brownian motion.  The ``correct" choice, taken from Lalley and Zheng~\cite{lalleyzheng10} (see below), is $\gamma=\frac{2d}{6-d}$ (so that for $d=2,3$, $R^\gamma=R^{d-1}\in\N$ if $R\in \N$). Let  $\Rightarrow$ denote weak convergence on the appropriate space, $M_F(\R^d)$ be the space of finite Borel measures on $\R^d$ with the topology of weak convergence, and $\mu(f)$ denote the integral  $\int fd\mu$ of a function $f$ with respect to a measure $\mu$. 
$C^2_K$ is the space of $C^2$ functions from $\R^d$ to $\R$ with compact support, and $D([0,\infty),M_F(\R^d))$ is the Skorohod space of cadlag $M_F(\R^d)$-valued paths.

\begin{conjecture} Assume $d\le 3$, $\gamma=\frac{2d}{6-d}$, and for appropriate 
$\{X_0^R\}$, $X_0^R\Rightarrow X_0$, for a compactly supported fixed finite measure $X_0$.
Then $X^R\Rightarrow X$ in $D([0,\infty),M_F(\R^d))$ as $R\to\infty$.  The limit $X$ is the unique solution to the martingale problem
\begin{equation}
 X_t(\phi) = \int \phi d X_t = X_0(\phi) + M_t(\phi) + \int_0^t X_s\left( \f{ \Delta \phi}6 + \theta \phi\right)ds - \int_0^t X_s(L_s\phi)ds \ , \ \forall \phi\in C^2_K,
 \label{eq:xtsuperbmdef}
\end{equation}
where $X$ is a continuous $M_F(\R^d)$-valued process, $L_s$ is the local time of $X$, i.e., satisfies $\int_0^t X_s(\psi) ds = \int L_t(x) \psi(x) dx$, and $M(\phi)$ is a continuous martingale with $\langle M(\phi)\rangle_t = \int_0^t X_s(\phi^2) ds$.
\end{conjecture}
The ``appropriate"  $\{X_0^R\}$ are those satisfying the regularity condition (17) in Theorem 2 of Lalley and Zheng~\cite{lalleyzheng10}.  That result establishes a weak convergence result which is very close to the above. Instead of $\Z^d_{R}$, they work with a ``village model" on $\Z^d\times\{1,\dots,N\}$ where sites $(x,m)$ and $(y,n)$ are neighbours iff $x$ and $y$ are nearest neighbours in $\Z^d$, or $x=y$ and $m\neq n$. The rescaling parameter in Theorem~2 of \cite{lalleyzheng10} is $N^{2/(6-d)}$.  Equating $N$ and $R^d$ so that the number of neighbours in the two models are of the same order of magnitude, leads to the scaling parameter $R^\gamma$ chosen above.  Xinghua Zheng in fact has proved the above Conjecture for $d=2$ [private communication].

The well-posedness of the above martingale problem is established, for example, in Theorem~2.2 of Lalley, Perkins, and Zheng\cite{lalleyperkinszheng} who also showed that for $d=2,3$, there is a phase
transition in \eqref{eq:xtsuperbmdef}:
\begin{thm}
Let $d=2$ or $d=3$.  There exists a $\theta_c = \theta_c(d)>0$ such that for all nonzero finite measures $X_0$, 
\begin{enumerate}[a.]
\item for $\theta<\theta_c$, $X$ becomes extinct.
\item for $\theta > \theta_c$, $X$ survives (with positive probability). 
\end{enumerate}
If $d=1$, then for all $\theta$ and all finite $X_0$, $X$ becomes extinct. 
\end{thm}
(Although percolation fails trivially for any $R$ and $p<1$ in $d=1$, there is some work to do to show this for the one-dimensional continuous model in the above.)  For $d=2,3$, a formal interchange of limits in $t$ and $R$ in the above results, and the equality $\frac{2d}{6-d}=d-1$ for $d=2,3$, leads us to the following refinement of M. Penrose's result \eqref{pcconv}:
\begin{conjecture}\label{pcR}
For $d=2$ or $3$, $V(R)p_c(R)-1\sim \frac{\theta_c}{R^{d-1}}$.
\end{conjecture}
(The analogue of the above for the ``village model" described above is raised in Section 2.2 of Lalley, Perkins, and Zheng\cite{lalleyperkinszheng}.) Simulations carried out by Deshin Finlay~\cite{finlay} suggest that 
\[V(R)p_c(R)-1\sim \begin{cases} 1.2/R  ,&\text{ if }d=2,\\
.7/R^2 , & \text{ if }d=3.
 \end{cases}
\]
The conjecture really should be made with $\frac{2d}{6-d}$ in place of $d-1$ even though
they are equal for the relevant values of $d$. For example, putting $d=4$ into
the former formula is consistent with the conjectured behaviour in the critical $d=4$ case cited above
(albeit with logarithmic corrections in $R$). 

The similarity with the critical infection rate of the long range contact process is discussed in Section 1.5 of \cite{lalleyperkinszheng}. Our $d\le 3$ setting corresponds to the $d=1$ setting of the long range contact process where the conjecture corresponding to Conjecture~\ref{pcR} remains unresolved (see the Conjecture following Theorem 2 in Durrett and Perkins~\cite{durrettperkins}).  In the contact process setting, however, upper and lower bounds establishing the correct rate of convergence (if not the exact constant) were established in Theorem 1 of Bramson, Durrett and Swindle\cite{crabgrass}.  Our main result, which adapts some of the nice ideas in the $1$-dimensional lower bound from \cite{crabgrass},  is a lower bound on $p_c(R)$, consistent with the rate in Conjecture~\ref{pcR}:
 \begin{thm}
For $d=2,3$ there is a constant $\theta = \theta(d)>0$, such that for all $R\in\N$
\[ V(R)p_c(R) - 1 \geq \frac{\theta}{R^{d-1}}.\]
\label{thm:lambdac.lowerbound}
\end{thm}

For the village model described above, independent work of E. Neuman and X. Zheng (in preparation) shows an analogue of the above result and a companion upper bound for $p_c$, using different methods.

In Section~\ref{sec:lowerbound-lambdac} we gather some results on the SIR epidemic $\eta_n$ and an associated branching random walk (BRW) $Z_n$.  In Subsection~\ref{sec:setup-increments-infected}  we give a more careful derivation of the connection between
bond percolation and $\eta_n$ (Lemma~\ref{lem:unexplored-independent}), and use it to derive an expression on the conditional increments of $|\eta_n|$, Corollary~\ref{cor:etainc}.  In Subsection~\ref{sec:envelope-brw} we couple $\eta$ with a dominating BRW $Z$ and in Subsection~\ref{sec:brw-estimate} we establish an upper bound on the probability that $Z$ exits a large ball by generation $n$.  The latter is classical (and goes back
at least to Dawson-Iscoe-Perkins\cite{dawsoniscoeperkins}) aside from the fact that the long range structure of the BRW means there is an additional parameter going to infinity. The changes needed to derive this result from the arguments in Section~4 of \cite{dawsoniscoeperkins} is presented in an Appendix.  In Section~\ref{extinctionpf} we will use Corollary~\ref{cor:etainc} and the above upper bound on the dominating BRW to show that for some fixed $\theta>0$ and $R$ satisfying $V(R)p(R)-1\ge \frac{\theta}{R^{d-1}}$ there is some $k$ so that $\E(|\eta_k|)<1$ (Proposition~\ref{meaneta}).  
We will then show, essentially by a comparison to subcritical branching (thanks in part to Lemma~\ref{lem:wipe.away.history}), that this gives a.s. extinction of $\eta$. In view of our assumption on $\lambda(R)$ this implies the lower bound on $\lambda_c(R)$.

\section{Preliminary Results}
\label{sec:lowerbound-lambdac}
In this section and for the remainder of the paper, we will consider an SIR epidemic $\eta$ on $\Z^d_R$ for $d=2$ or $3$, 
and choose the probability of infection $p = p(R)$, so that for a fixed $\theta\in(0,1]$,
\begin{equation}\label{binpar} \lambda(R)= V(R) p(R) =1 + \theta / R^{d-1}\ . 
\end{equation}

\subsection{The Increments of $\eta_n$}
\label{sec:setup-increments-infected}

Throughout this subsection we assume our initial condition $(\eta_0,\rho_0)$ is a fixed (deterministic) pair of disjoint finite sets of vertices. 
For each edge $(x,y)~\in~E(\Z^d_R)$, let
\[ n(x,y) = \inf \{ \ell \geq 1 : x\in \eta_{\ell-1} \text{ and } y\in \xi_{\ell-1}, \, \text{or }
x\in \xi_{\ell-1}\text{ and } y\in \eta_{\ell-1} \}\ . \]
(By convention, $\inf \emptyset = \infty$).  
Then $n(e)$ is the unique exploration time of the edge $e$ when $B(e)$ is used to define $\eta_n$ (and $n(e)=\infty$ means that the edge is never used in the definition of $\eta$).  Clearly $n(e)=n+1$ iff $e=(x,y)\in\p C_n$, and so $n(e)$ is a predictable $\calF_n$-stopping time, i.e., the events $\{ n(e) = n+1\}$ are $\calF_n$-measurable.  

 For $e\in E(\Z^d_R)$, define
\begin{equation}
 V_n(e) = \begin{cases}
B(e),& \text{ if }e\in \p C_{n-1},\\
2, & \text{ if }e\not\in \p C_{n-1}.
\end{cases}
\end{equation}
Then $x\in \eta_n$ if and only if there exists a sequence of points $x_0, \ \ldots\ , x_n = x$ such that $x_0\in \eta_0$, $x_k\notin\rho_0$ for all $1\le k\le n$, and $V_k(x_{k-1}, x_k)=1$ for all $k=1, \dots, n$.  In this way, $V_1, \dots, V_n$ describe the epidemic in terms of the geometry of the percolation substructure.  We define
\[ \bar \calF_n = \sigma( V_1, \dots, V_n) \ . \]
It is clear from the above that $\calF_n \subset \bar \calF_n$, and it is not hard to see that this inclusion is strict (although this will not be needed).  
\begin{lemma}
\label{lem:unexplored-independent}
(a) For any finite set of edges $S=\{(x_i,y_i):i\le m\}$ in $\Z_R^d$ and $\un S\subset S$,
\begin{align}\label{condindB}
\P(&B(e)=1\ \forall e\in\un S, B(e')=0\ \forall e'\in S\setminus\un S\,|\,\bar\calF_n)\\
\nonumber&=p(R)^{|\un S|}(1-p(R))^{|S\setminus\un S|}\text{  a. s. on the $\bar\calF_n$-measurable set } \{\p C_n=S\}.
\end{align}
(b) If $x\sim y$, then 
\begin{equation}\label{condexpB}\P(B(x,y) = 1\, |\,\bar\calF_n) = p(R)\text{ on the $\bar\calF_n$-measurable set $\{x\in \eta_n, y\in \xi_n\}$}.
\end{equation}

\end{lemma}

\begin{proof}
(a) Let $A_k: E(\Z^d_R)\to \{0,1,2\}$, $k=1,\dots, n+1$, be fixed functions such that each set
\[ S_k := \{ e\in E(\Z^d_R): A_k(e) \neq 2\},\quad k=1,\dots, n+1,\]
is finite.  (There are countably many such functions.)
Since $\bar \calF_n = \sigma(V_k, k\leq n)$, it suffices to show that
\begin{align}
\nonumber
&\P( B(e)=1\ \forall e\in\un S, B(e')=0\ \forall e'\in S\setminus\un S, \p C_n=S, (V_k)_{k\leq n} = (A_k)_{k\leq n}) \\
&= \P(B(e)=1\ \forall e\in\un S, B(e')=0\ \forall e'\in S\setminus\un S) \cdot \P(  \p C_n=S, (V_k)_{k\leq n} = (A_k)_{k\leq n} )
\label{eq:lem2.1independence}
\end{align}
The required result then follows by the independence of the $\{B(e):e\in E(\Z^d/R)\}$ and the definition of conditional probability. 

For simplicity of notation, let
\[ C =\{ \p C_n=S, V_k = A_k, k \leq n\}.\]
Without loss of generality, we may assume that $C$ is non-empty.    On the event $\cap_{k=1}^n\{V_k=A_k\}$, we know that $\p C_{k-1} = S_k$ for each $k =1,\dots, n$, since $V_k(e')$ is $\{0,1\}$-valued only on edges $e'\in \p C_{k-1}$, and is otherwise equal to $2$. Moreover, on $C$ we have $n(e)=n+1$ for all $e\in S$, and hence the sets
\[ S_1, \ \dots,\  S_{n}, S \]
are a collection of $n+1$ mutually disjoint sets.  Denote $S^{(n)} = \cup_{k\leq n} S_k$.  
On $\cap_{k=1}^n\{V_k=A_k\}$, for all $1\le k\le n$, $x\in\eta_k$ iff there is an $(x',x)\in S_k$ such that $A_k(x',x)=1$ (since on this set $A_k(x',x)=B(x',x)$).  Since  $y\in \xi_n$ iff $y\notin\rho_0\cup(\cup_{k=0}^n\eta_k)$, this means that on $\cap_{k=1}^n\{V_k=A_k\}$, we have $y\in\xi_n$ iff $y\notin \rho_0$ and for all $k\le n$ there is no $(y',y)\in S_k$ such that $A_k(y',y)=1$.  This shows that on $\cap_{k=1}^n\{V_k=A_k\}$,
\begin{align*}\partial C_n&=\{(x,y):x\sim y, y\notin\rho_0, \exists(x',x)\in S_n\text{ s.t. }A_n(x',x)=1,\\
&\phantom{\{(x,y):}\text{ and }\forall k\le n\text{ there is no }(y',y)\in S_k\text{ s.t. }A_k(y',y)=1\}\\
&=:S'_{n+1}.
\end{align*}
Therefore
\begin{equation*}\cap_{k=1}^{n+1}\{V_k=A_k\}=\begin{cases} \cap_{k=1}^n\{V_k=A_k\}\cap\{A_{n+1}(e)=B(e)\ \forall e\in S_{n+1}\},&\text{ if }S_{n+1}=S'_{n+1},\\
\emptyset, &\text{ if }S_{n+1}\neq S'_{n+1}.
\end{cases}
\end{equation*}
The obvious induction now shows that $\cap_{k=1}^n\{V_k=A_k\}\in \sigma(B(e):e\in S^{(n)})$. 
Also since $C\neq\emptyset$ we must have $S=S'_{n+1}$ and so $C=\cap_{k=1}^n\{V_k=A_k\}\in\sigma(B(e):e\in S^{(n)})$.
 Since $S\cap S^{(n)}$ is empty, and distinct edges have independent Bernoulli variables, \eqref{eq:lem2.1independence} is now immediate, and (a) is proved. 
\medskip

\noindent(b) Let $S$ be as in (a) and containing the edge $(x,y)$. Now sum the result in (a) over all $\bar S$ as in (a) with $(x,y)\in\bar S$ and $S$ fixed.  A simple application of the binomial theorem will give the conclusion of (b)  on the set $\{\partial C_n=S\}$.  Finally take the union over these events where $(x,y)\in S$ to derive (b).
\end{proof}

\begin{cor}\label{cor:etainc} 
\begin{equation}\label{eq:infected-recursive}
\E \Big[ |\eta_{n+1}|- |\eta_n| \big|  \bar \calF_n \Big] = \left( \f {\theta} {R^{d-1}} \right) |\eta_n| - \left( 1 + \f {\theta} {R^{d-1}} \right) \sum_{x\in \eta_n} \sum_{\substack{y \sim x \\ y\in \rho_n\cup\eta_n}} \f{ 1}{V(R)}.
\end{equation}
\end{cor}
\begin{proof}
Conditioned on the history of infected sites up to time $n$, the sites that are infected at time $n+1$ are the susceptible neighbours $y\in \xi_n$ of infected sites $x\in \eta_n$ such that $B(x,y)=1$.  Thus,
\begin{align*}
\E [|\eta_{n+1}|\ \big| \bar\calF_n] &= 
\E \left[ \sum_{x\in \eta_n} \sum_{\substack{ y\sim x}} \ind \{ y\in\xi_n \} \ind \{ B(x,y)=1\} \Big |\bar\calF_n \right] \\
&= \sum_{x\in \eta_n} \sum_{\substack{ y\sim x}}(1- \ind \{ y\in\rho_n\cup\eta_n \}) \E \left[ \ind \{ B(x,y) = 1 \} | \bar\calF_n \right]\\
&= \sum_{x\in \eta_n} \sum_{y\sim x} (1 - \ind \{ y\in\rho_n\cup\eta_n \} ) \cdot p(R) & (\text{by Lemma \ref{lem:unexplored-independent}}(b))\\
&= p(R) V(R) |\eta_n| - p(R) V(R) \sum_{x\in \eta_n} \sum_{\substack{y\sim x\\ y\in \rho_n \cup \eta_n}} \f 1 {V(R)},
\end{align*}
and so the result follows upon recalling $p(R)V(R)=\left( 1 + \f {\theta} {R^{d-1}} \right)$.
\end{proof}
The second term on the right-hand side of \eqref{eq:infected-recursive}  is an \textit{interference term} arising from the conditional expectation of those sites where an infection was attempted but suppressed as the site was already infected or recovered.  

\subsection{Coupling with a branching envelope}
\label{sec:envelope-brw}

Throughout this subsection we assume our initial condition $(\eta_0,\rho_0)=(\{0\},\emptyset)$.
As in Lalley-Zheng\cite{lalleyzheng10}, we may couple a copy of $\eta$ with a dominating branching random walk in our long range setting. To see this, label potential particles in the branching random walk by $I=\cup_{n=0}^\infty \calN(0)^n$, where $\calN(0)^0=\{\emptyset\}$.  It will be convenient to totally order $\calN(0)$ as $\{e_1,e_2,\dots,e_{V(R)}\}$ and then totally order each $\calN(0)^n$ lexicographically by $<$. For $\alpha=(\alpha_1,\dots,\alpha_n)\in \calN(0)^n\subset I$, write $|\alpha|=n$, $\alpha|i=(\alpha_1,\dots,\alpha_i)$ for $1\le i\le n$ ($|\emptyset|=0$), and $\pi\alpha=(\alpha_1,\dots,\alpha_{n-1})$ be the parent of $\alpha$, where if $n=1$ the parent is the root index $\emptyset$.  Concatenation in $I$ is denoted by $(\alpha_1,\dots,\alpha_n)\vee(\beta_1,\dots,\beta_m)=(\alpha_1,\dots,\alpha_n,\beta_1,\dots,\beta_m)$. Let $\{B^\alpha:\alpha\in I\setminus\{\emptyset\}\}$ be an iid collection of Bernoulli random variables with $\P(B^\alpha=1)=1-\P(B^\alpha=0)=p(R)$, and set $
\calG_n=\sigma(\{B^\alpha:|\alpha|\le n\})$.  The intuition is that each $\alpha\in I$ with $|\alpha|=n$ labels a {\it potential} individual in generation $n$ and $C^\alpha=\{e:B^{\alpha\vee e}=1\}$ are the locations of the children of $\alpha$ relative to the position of the parent, $\alpha$.  Let $M^\alpha=|C^\alpha|$ be the number of children of $\alpha$, so that $\{M^\alpha:\alpha\in I\}$ is a collection of independent binomial $(V(R),p(R))$ random variables. Moreover $\{M^\alpha:|\alpha|=n\}$ is jointly independent of $\calG_n$.  So for $n\ge 1$, each $\alpha\in I$ with $|\alpha|=n$ labels an individual alive in generation $n$ iff $B^{\alpha|i}=1$ for all $1\le i\le n$, in which case we write $\alpha\approx n$. Note that $\emptyset\approx 0$ always and for $|\alpha|=n$,
\[\{\omega:\alpha\approx n\}\in\calG_n.\]
If for each $\alpha\in I$, 
\begin{equation*}
Y^\alpha=\begin{cases}
\sum_{i=1}^{|\alpha|}\alpha_i\in\Z^d/R&\text{ if }\alpha\approx|\alpha|\\
\Delta&\text{ otherwise},
\end{cases}
\end{equation*}
then $Y^\alpha$ is the location of the particle $\alpha$ if it is alive, and is set to the cemetery state, $\Delta$, otherwise. 
In this way $Z_n=\sum_{\alpha\approx n}\delta_{Y^\alpha}$ ($n\ge 0$) defines the empirical distribution of a branching random walk (BRW) in which each individual in generation $n$ produces a binomial $(V(R),p(R))$ number ($M^\alpha$) of children whose positions, relative to their parent and given $\sigma(M^\alpha:|\alpha|=n)\vee\calG_n$, are uniformly distributed over
\[\calN(0)^{(M^\alpha)}=\{(e'_1,\dots,e'_{M^\alpha}):e'_i\in\calN(0), e'_1,\dots,e'_{M^\alpha} \text{ distinct}\}.\]
Note that conditional on $\sigma(M^\alpha:|\alpha|=n)\vee\calG_n$, the steps of the siblings from their common parent in generation $n$ are dependent but steps corresponding to distinct generation $n$ parents are independent. We set $Z_n(x)=Z_n(\{x\})$ for $x\in\Z^d/R$. Clearly $Z$ satisfies the natural Markov property with respect to $(\calG_n)$ and $Z_0=1_{\{0\}}$.  

We next define our coupled SIR epidemic $(\eta_n,\xi_n,\rho_n)$ inductively in $n$ so that $\eta_j(x):=1(x\in\eta_j)\le Z_j(x)$ for  all $x\in \Z^d_R$ and $j\le n$, $\calF_n:=\sigma(\rho_0,\eta_1,\eta_2,\dots,\eta_n)\subset\calG_n$, and $(\eta_j,\xi_j,\rho_j)_{j\le n}$ has the law of an SIR epidemic process. Set $\eta_0=1_{\{0\}}(=Z_0)$, assume the above for $n$, and consider $n+1$. Let $Y_n=\{y\in\Z^d/R:\exists x\in\eta_n\  \text{ s.t. }(x,y)\in\p C_n\}\subset\xi_n$. If $x\in \eta_n$, then $1\le Z_n(x)$ (by induction) and so we may choose a minimal $\alpha^x_n$ (with respect to our total order) in the non-empty set $\{\alpha\approx n:Y^\alpha=x\}$. One easily checks that $\alpha^x_n$ (set it equal to $\Delta$ if $x\notin \eta_n$) is $\calG_n$-measurable. We define (recall that $D_n(y)$ is as in \eqref{eq:Dndef})
\begin{equation}\label{etadefn}
\eta_{n+1}=\{y\in Y_n:\exists x\in D_n(y)\text{ s.t. }B^{\alpha^x_n\vee(y-x)}=1\},\ \xi_{n+1}=\xi_n\setminus\eta_{n+1},\ \rho_{n+1}=\xi_n\cup\eta_n.
\end{equation}
In this way $\alpha^x_n$ labels the BRW representative at $x$ in generation $n$ for $x\in\eta_n$. The fact that $D_n(y)$ and $Y_n$ are $\calF_n$-measurable and $\alpha_n^x$ is $\calG_n$-measurable shows that $\eta_{n+1}$ is $\calG_{n+1}$-measurable and so $\calF_{n+1}\subset\calG_{n+1}$.  Assume next that $y\in\eta_{n+1}$. Therefore there is an $x\in D_n(y)$ such that $B^{\alpha^x_n\vee(y-x)}=1$.  As we have $\alpha_n^x\approx n$, the latter implies that $\alpha_n^x\vee (y-x)\approx n+1$ and 
\[Y^{\alpha_n^x\vee(y-x)}=Y^{\alpha_n^x}+y-x=y.\]
Therefore $Z_{n+1}(y)\ge 1$ and we have proved that $\eta_{n+1}\le Z_{n+1}$.  To complete the induction it suffices to establish \eqref{etalaw} in the stronger form
\begin{equation}\label{stetalaw}
\P(\eta_{n+1}=V|\calG_n)= \prod_{y\in V_2\setminus V}(1-p)^{|D_n(y)|}\prod_{y\in V}[1-(1-p)^{|D_n(y)|}]\text{ a.s. on }\{\p C_n=S\},
\end{equation}
where $S,V$, and $V_2$ are as in \eqref{etalaw}. By \eqref{etadefn} the left-hand side of the above equals
\begin{align*}\P(&\forall y\in V_2\setminus V\  \forall x\in D_n(y)\ B^{\alpha_n^x\vee(y-x)}=0,\\
&\qquad\text{and }\forall y\in V\ \exists x\in D_n(y)\ \text{ s.t. }B^{\alpha_n^x\vee(y-x)}=1|\calG_n).
\end{align*}
One easily checks that the $\calG_n$-measurable collection of superscripts on the Bernoulli variables are distinct and label particles in generation $n+1$.  The fact that $\{B^\alpha:|\alpha|=n+1\}$ are jointly independent and independent of $\calG_n$ now gives \eqref{stetalaw} from the above. 

We restate what we have shown:

\begin{prop}\label{brenv}
There is a BRW $(Z_n,n\in \Z_+)$ and an SIR epidemic $(\eta_n,n\in\Z_+)$ on a common probability space s.t. $\eta_n(x)\le Z_n(x)$ for all $x\in\Z^d/R$, $n\in \Z_+$, and $Z_0=\eta_0=1_{\{0\}}$. Moreover $Z$ is a BRW in which each parent independently gives birth to a binomial $(V(R),p(R))$ number of children ($M$), where conditional on $M$, the offspring locations relative to their parent is uniform over $\calN(0)^{(M)}$. In addition, both $Z$ and $(\eta,\rho)$ satisfy their natural Markov properties with respect to a common filtration $(\calG_n)$ (for $(\eta,\rho)$ it is \eqref{stetalaw} above).
\end{prop}

\subsection{Branching random walk bounds}
\label{sec:brw-estimate}
We will need a pair of bounds on the BRW $Z$ constructed above which are in the literature either explicitly or implicitly. The first is bound on the survival probability for a sequence of Galton-Watson processes, which is almost immediate from Lemma~2.1(a) of  Bramson et al~\cite{crabgrass}.

\begin{lemma}\label{brwsurvival} Let $\{X^{(k)}\}$ be a sequence of Galton-Watson branching processes each starting with  a single particle. Assume $X^{(k)}$ has a Binomial ($N_k,q_k)$ offspring law where for some $C>0$, and large enough $n$,\\
\noindent (i)  $q_k\le 1/2$,\quad (ii) $1\le N_kq_k\le 1+C/k.$\\
Then 
\begin{equation}\label{eq:slightlysupercritical-survival}
\limsup_{k\to\infty}kP(X^{(k)}_k>0)\le K_{\ref{brwsurvival}}(C)=\frac{4C}{1-e^{-C}}.
\end{equation}
\end{lemma}
\begin{proof} If $Y_k$ has a binomial distribution with parameters $(N_k,q_k)$,
then an easy calculation shows that for large enough $n$
\[E(Y_k^3)\le 2\left(1+\frac{C}{k}\right)^3\le 2(1+C)^3.\]
This implies the uniform integrability of $\{Y_k^2\}$ which is needed to apply Lemma~2.1 of \cite{crabgrass}. Note also that the variance of $Y_k$ is $N_kq_k(1-q_k)\ge 1/2$ for large enough $k$, and so we have the lower bound on the variance also required in Lemma~2.1 of \cite{crabgrass}.  Finally the parameter $\nu_k:=(1/2)E(Y_k(Y_k-1))$ satisfies $\liminf_k\nu_k\ge 1/4$.  It is now straightforward to apply Lemma~2.1(a) of \cite{crabgrass} to get the above upper bound.
\end{proof}

The other result we will need concerns the range of the branching random walk $Z$ constructed in
Proposition~\ref{brenv}. Recall that $Z_0(x)=1(x=0)$ and the binomial offspring distribution has 
parameters $V(R)$ and $p(R)$ satisfying \eqref{binpar}.
Let $\calR_n$ denote the range of  $Z$ up until $n$, that is, 
$$\calR_n=\{x\in\Z^d_R:\sum_{j=1}^nZ_j(x)>0\}$$.

\begin{lemma}
\label{lemma:brw-range-A/r^2}
For any $c>0$ and $K\in\N$, there is an $A(c,K)$, non-decreasing in $c$ and $K$, such that for all $R\in\N$, $n\le cR^{d-1}$, and  $r\le K\sqrt n$, 
\[ \P(\calR_n\cap([-r,r]^d)^c\neq\emptyset)\le A (r+1)^{-2}.\]
\end{lemma}
Finer results are available for nearest neighbour branching random walks (see, e.g.,  Theorem~7 of Le Gall and Lin~\cite{legall}) but  as our branching random walk has some nonstandard features (such as a ``long range" random walk component) we outline the proof in the Appendix.  It is a straightforward modification of the results in Section 4 (particularly Lemma~4.9) of  Dawson, Iscoe, Perkins~\cite{dawsoniscoeperkins} from their branching Brownian motion setting to the long-range branching random walk setting here.

\section{Proof of Extinction}\label{extinctionpf}

Assume that $\eta_0=\{0\}$ and $\rho_0=\emptyset$ throughout this section. 
Recall that $p(R)$ is chosen so that \eqref{binpar} holds.  Our aim in this section is to establish Theorem~\ref{thm:lambdac.lowerbound} by showing $\eta$ becomes extinct a.s. for some positive value of $\theta$.  The first and main step is the following:
\begin{prop}\label{meaneta} There is a $\theta_0>0$ so that for all $R\in\N$, and all $0<\theta\le \theta_0$, there is a $k\in\{1,2,\dots,R^{d-1}+1\}$ so that $\E(|\eta_k|)\le 1-\theta$.
\end{prop}

\begin{proof} To shorten notation we set
\begin{equation}\label{epsdefn}
\eps=\frac{\theta}{R^{d-1}},
\end{equation}
where $\theta\in(0,1]$. We see that \eqref{eq:infected-recursive} implies
\[ \E \Big[ |\eta_{k+1}| - |\eta_k| \Big] \leq \E\left[\eps |\eta_k| - \sum_{x\in \eta_k}\sum_{\substack{ y\sim x \\ y\in \eta_k \cup \rho_k}}\f 1 {V(R)}\right]  \ . \]
 Since $|\rho_{n+1}| = \sum_{k=0}^{n} |\eta_k|$ and $|\eta_0|=1$, we can sum the above equation from $k=0$ to $n$ to get a telescoping sum
\begin{equation}\label{meanbnd1}\E \Big[ |\eta_{n+1}| - 1  \Big] \leq \E\Big[\eps |\rho_{n+1}| - \sum_{k=0}^n \sum_{x\in \eta_k} \sum_{\substack{y\sim x \\ y\in\eta_k \cup \rho_k}} \f 1 {V(R)}\Big] \ . 
\end{equation}

 Let us define the \textit{infection time} $\tau(x)$ of a site $x\in \Z^d$ by
\[
\label{eq:tau-infection-time}
\tau(x) = \begin{cases} 
n, & \text{ if }x\in \eta_n, \\
\infty, & \text{ if $x$ is never infected}. \end{cases}
\]
The second sum in \eqref{meanbnd1} contains only $x,y\in\rho_{n+1}$ so that $\tau(y)\le \tau(x)$.  The $x$'s can appear at most once in this summation but the $y$'s may appear multiple times corresponding to distinct values of $k$. Nonetheless we have the inequality
\begin{align}\label{interactionbnd1}
\sum_{k=0}^n \sum_{x\in \eta_k} \sum_{\substack{y \sim x \\y\in \rho_k \cup \eta_k}} \f 1 {V(R)} 
\ge \sum_{x\in \rho_{n+1}} \sum_{\substack{ y \in \rho_{n+1}\\ y\sim x \\ \tau(y)\leq \tau(x)}} \f 1 {V(R)} \ . 
\end{align}
By symmetry, we know that the sum on the right restricted to $\tau(y)<\tau(x)$ equals the same sum but now over 
$\tau(y)>\tau(x)$, and therefore,
 \[ \sum_{x\in \rho_{n+1}} \sum_{\substack{ y \in \rho_{n+1}\\ y\sim x \\ \tau(y)\leq \tau(x)}} \f 1 {V(R)} \geq \f 1 2 \sum_{x\in \rho_{n+1} } \sum_{\substack{y \sim x \\ y \in \rho_{n+1}}} \f 1 {V(R)} \ . \]
This, together with \eqref{meanbnd1} and \eqref{interactionbnd1}, implies
\begin{equation}
\label{eq:range-interference}
\E \Big[ |\eta_{n+1}| - 1 \Big] \leq \E\Big[\eps |\rho_{n+1}| - \f 1 2 \sum_{x\in \rho_{n+1}} \sum_{\substack{ y\sim x \\ y\in \rho_{n+1}}} \f 1 {V(R)}\Big] \ . 
\end{equation}
We can now proceed in a manner similar to that of Bramson \textit{et al.}~\cite{crabgrass}, and decompose the interference terms into regions of high density and low density.  Fix $K\in\N$, and consider the following sets:
\begin{align}
&I_n = [-K\sqrt n, K \sqrt n]^d \ , 
\vspace{0.15cm}\\
&A_n = \{ x\in \rho_{n+1} : |\{y\sim x:y\in\rho_{n+1}\}| \geq 6 \eps V(R)  \} \ .
\label{eq:InAndefn}
\end{align}
Then, as
\begin{align*}
 \sum_{x\in \rho_{n+1}} \sum_{\substack{y \sim x \\ y \in \rho_{n+1}}} \f 1 {V(R)} &\geq
  \sum_{x\in A_n} \sum_{\substack{y \sim x \\ y \in \rho_{n+1}}} \f 1 {V(R)} \\
 &\geq |A_n| 6 \eps \ , 
\end{align*}
we have from \eqref{eq:range-interference} and an elementary argument,
\begin{equation}
\label{eq:range-decomposition}
\E \Big[ |\eta_{n+1}| - 1  \Big] \leq \eps \E\Big[ |\rho_{n+1} \cap I_n^c | + |(\rho_{n+1}\setminus A_n)\cap I_n|- 2 |A_n \cap I_n| \Big].
\end{equation}
For the remainder of this section, we will work on each of the terms in \eqref{eq:range-decomposition} to show that the right-hand side of \eqref{eq:range-decomposition} is negative some $n$ for small enough $\theta$.  Henceforth we assume that 
\begin{equation}\label{ncond1}
n\le 4R^{d-1}.
\end{equation}


To prove that the term $|\rho_{n+1}\cap I_n^c|$ in \eqref{eq:range-decomposition} is small for large $K$, we will compare the range of the epidemic with the range of the  branching random walk $Z_n$ in Proposition~\ref{brenv}.  Recall that $Z_n$ is constructed on the same probability space as the epidemic such that $|\eta_k\cap A| \leq Z_k(A)$ for any set $A \subset \Z^d/R$.  A routine calculation for the branching random walk shows that if $X_k=X^R_k$ is a simple random walk, starting at the origin, taking steps uniformly in $\{x: x\in \Z_R^d, 0<\pnorm x \infty \leq 1\}$, we have
\begin{equation}\label{meanmeas} \E (Z_k(A)) = \left( 1 + \f {\theta} {R^{d-1}} \right)^k \P(X_k \in A) \ . \end{equation}
(Recall that the offspring steps at distinct times are independent in the BRW.)
Therefore, we have
\begin{align}
\nonumber \E (|\rho_{n+1} \cap I_n^c|) &= \sum_{k=0}^n \E (|\eta_k \cap I_n^c|) \\
\nonumber&\leq \sum_{k=0}^n \E (Z_k(I_n^c)|) \\
\nonumber&= \sum_{k=0}^n \left( 1 + \f {\theta} {R^{d-1}} \right)^k \P(X_k \in I_n^c) \\
\label{outInbnd}&\le e^4\sum_{k=1}^n\P(X_k\in I_n^c)\ ,
\end{align}
the last by our choice of $n$ in \eqref{ncond1}. 
We will get the desired bound, independent of $R$, by an application of the Azuma--Hoeffding inequality (see, for example Theorem~2 of Hoeffding~\cite{hoeff} and the comments at the end of Section 2 of that reference).
\begin{lemma}
\label{lem:azuma-hoeffding}
(Azuma-Hoeffding).  Let $M_n$ be a martingale with increments satisfying $|M_k - M_{k-1}|\le c_k$ a.s. for $k\in \N$.  Then we have for any $N\in \N$ and $D>0$,
\[ \P ( |M_N - M_0| \geq D ) \leq 2 \exp \left( \f{ - D^2}{2 \sum_{k=1}^N c_k^2} \right) \ . \] 
\end{lemma}
Applying the above to the martingale $(X_n^i)_n$, the $i$-th component of the random walk $(X_n)$, we get that for any $1\le k\le n$,
\begin{align*}
\P(|X_k^i| \geq K \sqrt n ) &\leq 2 \exp \left( \f{ - K^2  n}{ 2 k} \right) \leq 2 \exp(-K^2/2) \ .
\end{align*}
Therefore, we have for $1\le k\le n$,
\begin{equation}
\label{eq:outsidebox-prob}
 \P(\Vert X_k\Vert_\infty \geq K  \sqrt n) \leq 2 d \exp(-K^2/2),
 \end{equation}
which by \eqref{outInbnd} yields
\begin{equation}
\label{eq:outsidebox-expected}
\E( |\rho_{n+1} \cap I_n^c| )\le e^4 2dn \exp(-K^2/2) \ . 
\end{equation}

Now consider the second term in \eqref{eq:range-decomposition} and define $\zeta_n=(\rho_{n+1}\setminus A_n)\cap I_n$.  Let $C(r)$ denote a closed cube of $\ell^\infty$ diameter $1$ that is $\ell^\infty$ distance $r\ge 0$ from the origin in $\Z^d/R$.  There are at most $6 \eps V(R)$ elements inside $C(r) \cap \zeta_n$, since otherwise an element in $C(r) \cap \zeta_n$ would have more than $6\eps V(R)$ neighbours in $\rho_{n+1}$, contradicting the definition of $\zeta_n$.  Comparing the epidemic with the coupled branching random walk $Z_n$ with range $\calR_n$ up to time $n$, we can use Lemma \ref{lemma:brw-range-A/r^2} to conclude (recall \eqref{ncond1} and $\eps=\theta/R^{d-1}$) that for $r\le 2K\sqrt{n}$,
\begin{align}\label{mean cube}
\nonumber\E (|C(r)\cap \zeta_{n}|)\le \E(6\eps V(R)1(C(r)\cap\calR_n\neq\emptyset))&\le 6\eps V(R)A(4,2K)(r+1)^{-2}\\
&\le c_1(K)\theta R(r+1)^{-2}\ .
\end{align}
The number of such boxes $C(r)$ with ``integer corners" at $\ell^\infty$ distance $r\in[m,m+1]$ ($m\in\Z_+$) from the origin is bounded by $C'(m+1)^{d-1}$. We can cover $\zeta_n\subset I_n$ by the collections of these boxes for $m=0,\dots, \lfloor K\sqrt n\rfloor$ and conclude that for a constant $c_2=c_2(K)$,
\begin{align}
\nonumber \E (|\zeta_n| )&\leq \sum_{m=0}^{\lfloor K\sqrt n\rfloor} C' (m+1)^{d-1} c_1(K)\theta R(m+1)^{-2}\\
&\leq c_2(K)\theta R \sum_{m=0}^{\lfloor K\sqrt n\rfloor}  (m+1)^{d-3}.
\label{eq:zetan.d=2.d=3}
\end{align}
We strengthen \eqref{ncond1} and now assume
\begin{equation}\label{nbnd2}
R^{d-1}\le n\le 4R^{d-1}.
\end{equation}

Assume first that $d=3$.   Continuing from \eqref{eq:zetan.d=2.d=3}, we use $n^{-1/2}\le R^{-1}$ (from \eqref{nbnd2}) to see that for $n$ as above,
\begin{equation}\label{eq:zetan.d=3} \E( |\zeta_n| )\leq c_2(K)(K\sqrt n+1)\theta R\le 2c_2(K)Knn^{-1/2} \theta R \le 2c_2(K)K\theta n \ . 
\end{equation}
Substitute \eqref{eq:zetan.d=3} and \eqref{eq:outsidebox-expected} into \eqref{eq:range-decomposition} to get 
\begin{equation} \label{etabnd2}
\E \Big[ |\eta_n| - 1 \Big] \leq \eps n \left[ 2e^4d \exp(-K^2/2) + 2c_2(K) K\theta  -  \f 2 n\E(|A_n \cap I_n|) \right] \text{ for }R^2\le n\le 4R^2\ . 
\end{equation}
Now set $n=R^2$ and suppose that $\E(|\eta_k|) \geq 0.9$ for $k=1, \dots, n$.    Then, as $|\rho_{n+1}| = \sum_{k=0}^n |\eta_k|$, we have 
\begin{equation}\label{rholb}\E(|\rho_{n+1}|) \geq 0.9n.
\end{equation}  
First choose $K$ large enough so that $2e^4d \exp(-K^2/2) < 0.01$, and then $\theta>0$ small enough so that $2c_2(K)K\theta <0.01$.  This and the bounds \eqref{eq:zetan.d=3} and \eqref{eq:outsidebox-expected} show that 
\begin{align*}
0.9n\le \E(|\rho_{n+1}|)&\le \E(|\rho_{n+1}\cap I_n^c|+|(\rho_{n+1}\setminus A_n)\cap I_n|+|A_n\cap I_n|)\\
&\le n[e^42de^{-K^2/2}+2c_2(K)\theta K+\frac{1}{n}\E(|A_n\cap I_n|)]\\
&\le n\times.02+\E(|A_n\cap I_n|),
\end{align*}
and so $\E(|A_n\cap I_n|)\ge .88n$.  
Therefore inserting the above and our choices of $n$, $K$ and $\theta>0$ into \eqref{etabnd2}
we arrive at 
\[\E(|\eta_n|-1)\le\theta[.02-1.76]\le -\theta.\]
Recall we had assumed that $\E(|\eta_k|)\ge 0.9$ for $k=1,\dots,n=R^2$, and so we may
conclude that in any case for small enough $\theta>0$ as above,  for all $R\in\N$,
\begin{equation}\label{etabndd=3}\text{for some }1\le k\le R^{2}+1,\ \E(|\eta_k|-1)\le -(.1\wedge \theta)<0.
\end{equation}
This completes the proof for $d=3$ as we may take $\theta\le .1$.

Assume next that $d=2$. If one proceeds in the same manner as in the $d=3$ case above, the sum in \eqref{eq:zetan.d=2.d=3} will lead to  an extra logarithmic factor in the lower bound of $\lambda_c$.  In this case 
we will need to improve the bound \eqref{mean cube} for smaller values of $r$ by an appeal to the local central limit theorem. 

Continue to assume \eqref{nbnd2}, now with $d=2$.  Let $\{U^R_i:i\in\N\}$ be iid random vectors which are uniformly distributed over the $V(R)$ points in $\calN(0)$.  We can couple this sequence an iid sequence, $\{U_i:i\in\N\}$ of random vectors which are uniformly distributed over $[-1,1]^2$ and such that $\Vert U^R_i-U_i\Vert_\infty\le R^{-1}$ (the precise assignment of regions of $[-1,1]^2$ to points in $\calN(0)$ is not important).  If $S_k=\sum_{i=1}^k U_i$ and $X^R_k=\sum_{i=1}^k U^R_i$, then by the triangle inequality,
\begin{equation}\label{stepcouple}
\text{for }k\le n,\ \Vert S_k-X_k^R\Vert_\infty\le k/R\le n/R\le 4.
\end{equation}
If $C(r)$ is as above and $\bar C(r)$ is a box with the same centre but with $\ell^\infty$ diameter $9$, then by  our comparison result with BRW and \eqref{meanmeas} we have
\begin{align}\label{zetad=2}
\nonumber\E(|\zeta_n\cap C(r)|)&\le \E\Bigl(\sum_{k=0}^nZ_k(C(r))\Bigr)\\
\nonumber&\le e^{\theta n/R}\sum_{k=0}^n\P(X^R_k\in C(r))\\
&\le e^{4\theta}\left[\sum_{k=1}^n\P(S_kk^{-1/2}\in k^{-1/2}\bar C(r))+1_{C(r)}(0)\right],
\end{align}
the last by \eqref{stepcouple} and \eqref{nbnd2}. In the 
last line $k^{-1/2}\bar C(r)=\{k^{-1/2}x:x\in\bar C(r)\}$.  Let $q_k$ be the density of $S_k/\sqrt{k}$, $\sigma^2=1/3$ (the variance of the uniform law on $[-1,1]$) and $\Vert (x_1,x_2)\Vert_2^2=x^2_1+x_2^2$.  The local central limit theorem (e.g. (19.26) in 
Bhattacharya and Rao \cite{bharao}) implies
\begin{equation}
\limity n \sup_{x\in \R^d} (1+\Vert x\Vert_\infty^2) |q_n(x) - (2\pi\sigma^2)^{-1}\exp(-\Vert x\Vert_2^2/(2\sigma^2))| = 0.
\label{lclt}
\end{equation}
Integrate the above bound to conclude that for some $C$,
\begin{align*}
\sum_{k=1}^n\P(S_kk^{-1/2}\in k^{-1/2}\bar C(r))&\le C\sum_{k=1}^n\int_{\bar C(r)k^{-1/2}}(1+\Vert x\Vert_\infty^2)^{-1}\,dx+\int_{\bar C(r)k^{-1/2}}\exp(-\Vert x\Vert_2^2)\,dx\\
&\le C\sum_{k=1}^n(9 k^{-1/2})^2\Bigl[\frac{1}{1+((r-4)^+)^2/k}+\exp(-((r-4)^+)^2/k)\Bigr]\\
&\le C\sum_{k=1}^n(k+r^2)^{-1}\\
&\le C\left[\log\Bigl(\frac{n+r^2}{1+r^2}\Bigr)+\frac{1}{1+r^2}\right]\\
&\le C\left(\log\Bigl(1+\frac{2n}{(r+1)^2}\Bigr)+\frac{2}{(r+1)^2}\right)\le C\log\Bigl(1+\frac{2n}{(1+r)^2}\Bigr),
\end{align*}
the last by some calculus. Use this in \eqref{zetad=2} and recall that $\theta\le 1$ to conclude that 
\begin{equation}\label{zeta2}
\E(|\zeta_n\cap C(r)|)\le C\log\Bigl(1+\frac{2n}{(1+r)^2}\Bigr).
\end{equation}

We consider cubes  of the form $C(r)=[j_1,j_1+1]\times[j_2,j_2+1]$ for $j=(j_1,j_2)\in\Z^2$ such that $C(r)\cap I_n\neq\emptyset$.  Recalling that $r$ is the $\ell^\infty-$distance of $C(r)$ from $0$ we see that $r\le K\sqrt n$.  If $m\in\{1,2\dots,\lceil\sqrt nK\rceil\}$, the number of such cubes $C(r)$ with $r\in[m-1,m]$ is bounded by $c_0m$ for some $c_0$.  
As $m\le r+1$, \eqref{zeta2} and \eqref{mean cube} imply that
\begin{equation}\label{zeta3}
\E(|\zeta_n\cap C(r)|)\le\min\left(\frac{c_1(K)\theta R}{ m^{2}},C\log\Bigl(1+\frac{2n}{m^2}\Bigr)\right).
\end{equation}
Let $\delta=\theta\log(1/\theta)$ where we now take $0<\theta<e^{-1}$, and note that $\delta\le e^{-1}$.  Set $$M_1=\{m\in\N: m\le \lceil \sqrt nK\rceil:nm^{-2}\le \delta^{-1}\} \text{ and }M_2=\{m\in\N: m\le \lceil \sqrt nK\rceil:nm^{-2}> \delta^{-1}\}.$$  The obvious covering argument and \eqref{zeta3} show that 
\begin{align*}
\E(|\zeta_n|)&\le \sum_{m=1}^{\lceil \sqrt{n}K\rceil}c_0m \min\left(\frac{c_1(K)\theta R}{ m^{2}},C\log\Bigl(1+\frac{2n}{m^2}\Bigr)\right)\\
&\le c_3(K)\left[\sum_{m\in M_1}\theta Rm^{-1}+\sum_{m\in M_2}m\log\Bigl(1+\frac{2n}{m^2}\Bigr)\right]\\
&\le c_3(K)\left[\theta n(1+\log((K+1)/\sqrt\delta))+n\sum_{m\in M_2}\frac{m}{\sqrt n}\log\Bigl(1+\frac{2n}{m^2}\Bigr)\frac{1}{\sqrt n}\right]\text{ (recall $R\le n$)}\\
&\le c_3(K)n\left[\theta\log((K+1)/\sqrt\delta))+\int_0^{\sqrt\delta}u\log(1+2u^{-2})\,du\right].
\end{align*}
The last line follows by a bit of calculus, and a bit more gives the bound
\begin{align*}\E(|\zeta_n|)&\le c_3(K)n\left[\theta\log((K+1)/\sqrt\delta))+\delta\log(1/\delta)\right]\\
&\le c_3(K)\theta n\left[\log((K+1)/\sqrt\delta))+\log(1/\theta)\log(1/\delta)\right]\\
&\le c_3(K)\theta n[\log((K+1)/\theta)]^2.
\end{align*}
Use the above bound in place of \eqref{eq:zetan.d=3}, so that instead of \eqref{etabnd2} we get 
\[E(|\eta_n|-1)\le\eps n[2e^4\exp(-K^2/2)+C(K)\theta[\log((K+1)/\theta)]^2-\frac{2}{n}\E(|A_n\cap I_n|)]\quad\text{ for }R\le n\le 4R.\]
Since $\theta[\log((K+1)/\theta)]^2$ decreases to $0$ as $\theta\downarrow 0$, we may now proceed just as for $d=3$ to conclude (in place of \eqref{etabndd=3}) that for small enough $\theta>0$, for all $R\in\N$,
\begin{equation}\label{etabndd=2}\text{for some }1\le k\le R+1,\ \E(|\eta_k|-1)\le -(.1\wedge \theta)<0.
\end{equation}
This completes the proof for $d=2$. 
\end{proof}

For $n\in\Z_+\cup\{\infty\}$, let 
\[L_n=\sum_{j=0}^n|\eta_j|=|\cup_{j=0}^n\eta_j|.\]
Clearly if $\rho_0=\emptyset$, then $L_n=|\rho_{n+1}|$.

\medskip
\noindent{\bf Proof of Theorem~\ref{thm:lambdac.lowerbound}.} Let $\theta=\theta_0>0$ and $k$ be as in Proposition~\ref{meaneta}. Let $c_k=\E(L_k)$, where it is understood that $(\eta_0,\rho_0)=(\{0\},\emptyset)$ under $\P$, as usual. By Proposition~\ref{brenv} and \eqref{binpar},
\[c_k\le \E\Bigr[\sum_{i=0}^k Z_i(1)\Bigl]\le \sum_{i=0}^k\Bigl(1+\frac{\theta}{R^{d-1}}\Bigr)^i<\infty.\]
Let $\{\beta_n\}$ be a Galton--Watson branching process with offspring law $\P(|\eta_k|\in\cdot)$ and initial state $\beta_0=1$.  We claim
\begin{equation}
\forall m \in \Z_+, \ \forall n \geq mk, \ \E( L_n) \leq c_k \sum_{j=0}^{m-1} \E(\beta_j) + \E(\beta_m) \E(L_{n-km}) \ . 
\label{eq:gw.branching}
\end{equation}
First, we  show that the claim would complete the proof.  Let $n=(m+1)k$ in \eqref{eq:gw.branching} and then let $m\to\infty$ to conclude that 
\[ 
\E(L_\infty)\le c_k\sum_{j=0}^\infty \E(\beta_j)+\lim_{m\to\infty}\E(\beta_m)c_k=c_k\sum_{j=0}^\infty \E(|\eta_k|)^j<\infty. \]
This implies that $\cup_{j=0}^\infty\eta_j$ is finite a.s. and so $\eta_j=\emptyset$ for $j$ large enough.  Recalling that extinction of $\eta$ starting at $\{0\}$ occurs iff the percolation cluster $\mathcal{C}_0$ is finite, we can conclude that $p_c(R)V(R)\ge 1+\frac{\theta_0}{R^{d-1}}$, and the proof is complete.

We prove \eqref{eq:gw.branching} by induction on $m$. The result is trivial for $m=0$.  Assume the result for $m$. 
Let $\{\eta^i_\cdot:i\in\N\}$ be independent and identically distributed copies of $\eta$ under $\P$, independent of the branching process $\beta$.  Let $L^i_n = \sum_{j\leq n}|\eta_j^i|$, and $ \calF_n^i=\sigma(\eta^i_k,k\le n)$ be the generated filtrations.  If $n\ge (m+1)k$, we can rewrite the last term in \eqref{eq:gw.branching} as
\begin{align*}
\sum_{i=1}^\infty\P(i\le\beta_m) \E(L^i_{n-mk}) &=\sum_{i=1}^{\infty}\P(i\le \beta_m)[\E(\E(L^i_{n-mk}-L^i_k \big| \calF_k^i) +L_k^i)] \\
&= \E \left[  \sum_{i=1}^{\beta_m} \E \left[ L_{n-km}^i- L_k^i\big| \calF_k^i \right] \right] + c_k \E(\beta_m) \\
&= \E \left[  \sum_{i=1}^{\beta_m} \E_{\eta_k^i, \rho_k^i} \left[ L_{n-(m+1)k}^i   \right] \right] + c_k \E(\beta_m) &(\text{by the Markov property \eqref{markovprop}})\\
&\leq \E \left[  \sum_{i=1}^{\beta_m} \E_{\eta_k^i, \emptyset} \left[ L_{n-(m+1)k}^i  \right] \right] + c_k \E(\beta_m) &(\text{by Lemma \ref{lem:wipe.away.history}}(a)).\\ 
\end{align*}
Lemma~\ref{lem:wipe.away.history}(b) implies that 
\[\E_{\eta_k^i, \emptyset} \left[ L_{n-(m+1)k}^i  \right]\le|\eta^i_k|\,\E(L_{n-(m+1)k}).\]
So substituting this into the previous display, we conclude that 
\begin{align*}\E(\beta_m)\E(L_{n-km})&\le\E\left[\sum_{i=1}^{\beta_m}|\eta^i_k|\right]\E(L_{n-(m+1)k})+c_k\E(\beta_m)\\
&=\E(\beta_{m+1})\E(L_{n-(m+1)k})+c_k\E(\beta_m).
\end{align*}
Put this into \eqref{eq:gw.branching} (the induction hypothesis), to see that \eqref{eq:gw.branching} holds for $m+1$, completing the induction, and hence the proof of the Theorem. 
\qed

\section{Appendix: Proof of Lemma \ref{lemma:brw-range-A/r^2}}

Consider $\mu$ particles starting at the origin in $\R^d$.  For each $i\in\Z^+$, on $[\frac{i}{\mu},\frac{i+1}{\mu})$ each particle follows an independent $d$-dimensional Brownian motion, and at time $\frac{i+1}{\mu}$ the particle is replaced by $0$ or $2$ offspring at the parent's location, each with probability $\frac{1}{2}$. Let $\hat X^\mu_t$ be the random measure which puts mass $\mu^{-1}$ at the location of each particle at time $t$.  (See Section 2 of \cite{dawsoniscoeperkins} for a more detailed description of this branching Brownian motion.) Let $\hat X$ be the super-Brownian motion which is the unique in law solution of the following martingale problem:
\begin{equation*}
 \hat X_t(\phi) = \phi(0) + \hat M_t(\phi) + \int_0^t \hat X_s\left( \f{\sigma^2}{2} \Delta \phi+\theta'\phi\right)ds  \ , \ \forall \phi\in C^2_K,  \qquad\qquad(MP)_{\sigma^2,\theta'}
\end{equation*}
where $\hat X$ is a continuous $M_F(\R^d)$-valued process, and $\hat M(\phi)$ is a continuous martingale with $\langle \hat M(\phi)\rangle_t = \int_0^t \hat X_s(\phi^2) ds$.

It is well-known that $\hat X^\mu$ converges weakly to $\hat X$ (with $\sigma^2=1$ and $\theta'=0$) on $D([0,\infty),M_F(\R^d))$ but we will need a result on the convergence of the ranges which does not follow from this alone.  Let $$\hat \calR_t^\mu=\{x\in\R^d: \exists s\le t \text{ s. t. }\hat X_s^\mu(\{x\})>0\},$$ and let $\hat \calR_t$ denote the closed support of $\int_0^t \hat X_s(\cdot)\,ds$. 
\begin{lemma}\label{limitranges} If $\P_{\delta_0}$ denotes the law of $\hat X$ as above with $\sigma^2=1$ and $\theta'=0$, then
\begin{equation}\limsup_{\mu\to\infty} \P(\hat \calR_1^\mu\cap((-1,1)^d)^c\neq\emptyset)\le \P_{\delta_0}(\hat \calR_1\cap((-1,1)^d)^c\neq \emptyset).
\end{equation}
\end{lemma}
\begin{proof} This is immediate from Lemma~4.9 and Theorem~4.7(a) of Dawson, Iscoe, and Perkins~\cite{dawsoniscoeperkins} and the Transfer Principle of nonstandard analysis (to translate into standard terms).  More specifically the first two results imply that for $\mu$ infinite (fixed), if $x\in\calR_1^\mu\cap((-1,1)^d)^c$, then there is a sequence $x_n$ in $\hat\calR_1$ converging to the standard part of $x$, $st(x)$.   This shows the latter must be in 
$\hat\calR_1$ as this set is closed.  Since $st(x)$ is also in the complement of $(-1,1)^d$, we have shown that $\hat\calR_1\cap((-1,1)^d)^c$ is non-empty.  An immediate application of the Transfer Principle now gives the required result.
\end{proof}

\begin{remark} Just using the above weak convergence and elementary properties 
of $\hat X$ (it never charges the boundary of $(-1,1)^d$) one can easily show that \[\liminf_{\mu\to\infty} \P(\hat \calR_1^\mu\cap((-1,1)^d)^c\neq\emptyset)\ge \P_{\delta_0}(\hat \calR_1\cap((-1,1)^d)^c\neq \emptyset),\]
but it is the upper bound that will be of interest. 
\end{remark}
Our immediate goal is to extend Lemma~\ref{limitranges} from the above branching Brownian motion to the context of the BRW $Z$ constructed in Section~\ref{sec:envelope-brw}.  We shall see that Lemma~\ref{lemma:brw-range-A/r^2} then follows easily.

To more closely parallel the setting in \cite{dawsoniscoeperkins} we modify the setup for the branching envelope $Z$ in Section~\ref{sec:envelope-brw} while constructing a BRW $\tilde Z$ with the same branching dynamics as $Z$. Let $\tilde I=\cup_{n=0}^\infty \N\times\{1,\dots,V(R)\}^n$ and for $\beta,\beta'\in\tilde I$, we define $|\beta|$, $\beta|i$, $\pi\beta$ and $\beta<\beta'$ as for $I$ in Section~\ref{sec:envelope-brw}. Assume $\{\tilde M^\beta:\beta\in\tilde I\}$ are iid Binomial $(V(R),p)$ random variables, denoting the number of offspring of particle $\beta$, where we assume $pV(R)\ge 1$ and, as always, $R\in\N$.  Fix an initial number of particles $\mu\in\N$. Write $\beta\approx n$ iff $|\beta|=n$, $\beta_0\le \mu$, and $\beta_{i+1}\le \tilde M^{\beta|i}$, for all $0\le i<n$, meaning that $\beta$ labels a particle which is alive in the $n$th generation. 
Next, let $(d^{\beta\vee i},i\le V(R))_{\beta\in\tilde I}$ be a collection of iid random vectors, each uniformly distributed over $\calN(0)^{(V(R))}=\{(e_1,\dots,e_{V(R)}): \{e_i\} \text{ all distinct}\}$. For each $\beta$ these are the displacements of the {\it potential} children from the parent $\beta$. Therefore the historical path followed by the ancestors of a particle $\beta\in\tilde I$ is
\[\tilde Y^{\beta,\mu}_t=\tilde Y^\beta_t=\sum_{i=1}^{|\beta|}1(i\le \lfloor\mu t\rfloor)d^{\beta|i},\]
and its current location is 
\[\tilde Y^\beta=\sum_{i=1}^{|\beta|}d^{\beta|i}\in\Z^d/R\quad(\text{so if }|\beta|=0, \text{ then }\tilde Y^\beta=0).\]
Note that for each $\beta$, $\tilde Y^\beta_\cdot$ is a random walk which jumps at times $i/\mu$ for $i\le |\beta|$, and whose step distribution is uniform over $\calN(0)$.  Let $\tilde \calF_n=\sigma(\tilde M^\beta:|\beta|<n)\vee\sigma(d^\beta:1\le|\beta|\le n)$.  Note for each fixed $|\beta|=n$, the event  $\{\beta\approx n\}$ is in $\tilde\calF_n$ and $\tilde Y^\beta$ is $\tilde\calF_n$-measurable. Therefore
\[ \tilde Z_n=\sum_{\beta\approx n} \delta_{\tilde Y^\beta}\text{ is an $\tilde\calF_n$-measurable random measure}.\]
Conditional on $\tilde\calF_n$, $(\tilde M^\beta:\beta\approx n\}$ are iid binomial $(V(R),p)$ random variables, and conditional on $\tilde \calF_n\vee\sigma(\tilde M^\beta:\beta\approx n)$, $(\tilde Y^{\beta\vee i}-\tilde Y^\beta:i\le \tilde M^\beta)_{\beta\approx n}=(d^{\beta\vee i}:i\le\tilde M^\beta)_{\beta\approx n}$ are independent random vectors which for each $\beta\approx n$ are uniformly distributed over $\calN(0)^{(\tilde M^\beta)}$. This shows $\tilde Z$ is a BRW with the same offspring law as that of the branching envelope $Z$, and hence:
\begin{align}\label{tildeZ} \text{if }&\text{$\mu=1$, the laws of $\tilde Z$ and $Z$ (from Proposition~\ref{brenv}) are identical,}\\
\nonumber& \text{and in general $\tilde Z$ is equal in law to a sum of $\mu$ iid copies of $Z$.}
\end{align}
Consider the rescaled random measures given by
\[\tilde X^\mu_t(A)=\frac{1}{\mu}\tilde Z_{\lfloor \mu t \rfloor}(\sqrt\mu A)=\Bigl(\frac{1}{\mu}\sum_{\beta\approx\lfloor\mu t\rfloor}\delta_{\tilde Y^\beta/\sqrt\mu}\Bigr)(A).\]
\begin{prop}\label{BRWwkconv} If $\mu_n\to\infty$ and we choose $R_n\to\infty$ and $p_n$ so that $p_nV(R_n)\ge1$ and \hfil\break$\lim_n\mu_n(p_nV(R_n)-1)=\theta'\ge 0$, then $\tilde X^{\mu_n}\Rightarrow \hat X$ in $D([0,\infty),M_F(\R^d))$, where $\hat X$ is the super-Brownian motion satisfying $(MP)_{\sigma^2, \theta'}$, with $\sigma^2=1/3$.
\end{prop}
This is a minor modification of the classical convergence theorem and may be proved by making minor changes in the proof, for example, in Chapter II of \cite{perkins02}. The value $\sigma^2=1/3$ arises as the variance of the marginals of the uniform distributions over $[-1,1]^d$ and the drift $\theta'$ arises since the mean number of offspring is $p_nV(R_n)\sim 1+\frac{\theta'}{\mu_n}$.  Note there is dependence between particle steps only if the particles are siblings and even here the steps are uncorrelated.  This leads only to very minor alterations to the usual proof in the setting of completely independent displacements.

We assume in the rest of this section that $(\mu_n,p_n,R_n)$, $\tilde X^{\mu_n}$, and $\hat X$ are as in Proposition~\ref{BRWwkconv}, $\P_{\delta_0}$ is the law of $\hat X$ and $\hat R_t$ is the closed support of $\int_0^t\hat X_s\,ds$.  Let 
\[\tilde \calR_t^{\mu_n}=\{x\in\R^d: \exists s\le t \text{ s. t. } \tilde X_s^{\mu_n}(\{x\})>0\}.\]
Here is the version of Lemma~\ref{limitranges} we will need.
\begin{lemma}\label{BRWlimitranges} 
\begin{equation}\limsup_{n\to\infty} \P(\tilde \calR_1^{\mu_n}\cap((-1,1)^d)^c\neq\emptyset)\le \P_{\delta_0}(\hat \calR_1\cap((-1,1)^d)^c\neq \emptyset).
\end{equation}
\end{lemma}
The result will follow just as in the proof of Lemma~\ref{limitranges}, once the analogues of Lemma~4.9 and Theorem~4.7 of \cite{dawsoniscoeperkins} are established.  The analogue of Theorem~4.7 will be immediate from the following uniform modulus of continuity for the historical paths of all particles in the BRW (just as Theorem~4.7 of \cite{dawsoniscoeperkins} follows from Theorem~4.5 of that reference).

\begin{lemma}\label{modcont} For each $L\in\N$ there are positive constants $c_i(L)$, $i=1,2,3$ and non-negative random variables, $\delta(L,\mu_n)$, such that for all $n$,
\[\P(\delta(L,\mu_n)\le \rho)\le c_1(L)\rho^{c_2(L)}\text{ for }0\le \rho\le c_3(L),\]
and if $s,t\in[0,L]$ satisfy $\mu_n^{-1}\le t-s\le \delta(L,\mu_n)$, then for all $\beta\approx \lfloor \mu_nt\rfloor$, 
\[|Y^{\beta,\mu_n}_t-Y^{\beta,\mu_n}_s|<(t-s)^{1/8}.\]
\end{lemma}
In Theorem~4.5 of \cite{dawsoniscoeperkins} the analogous result is stated for the branching Brownian motion with \break $c\sqrt{(t-s)\log(1/(t-s))}$ ($c>2$) in place of $(t-s)^{1/8}$ but any modulus function will do for our purposes.  The proof of the above lemma is very similar. In place of the Gaussian bounds for the Brownian paths one uses Lemma~\ref{lem:azuma-hoeffding} as the coordinates of $Y^\beta_t$ are martingales with bounded jumps.  The above cruder modulus helps handle the very small values of $t-s$ .  Note that the restriction $t-s\ge 1/\mu_n$ is natural as the modulus is really only needed for $s,t\in\{i/\mu_n:i\in \Z_+\}$.  
The slight super-criticality of our BRW also leads to some minor changes including the 
time cut-off $L$ and the use of the survival bound \eqref{eq:slightlysupercritical-survival} for our BRW (see the calculation in \eqref{highdensity} below).  We omit the proof. 

\medskip
\noindent{\bf Proof of Lemma~\ref{BRWlimitranges}.} As noted above, given the previous lemma, it suffices to establish the analogue of Lemma~4.9 in \cite{dawsoniscoeperkins} for our slightly supercritical rescaled BRW's $\tilde X^{\mu_n}$.  The proof for branching Brownian motion goes through unchanged using Lemma~\ref{modcont} once the analogue of Lemma~4.8 of \cite{dawsoniscoeperkins} is established so we now consider this result. It is a nonstandard formulation of the fact that if $t>0$ and $t_n\to t$ are fixed, then the sets $S^{\mu_n}_{t_n}=\{x\in\Z^d/R_n:\hat X_{t_n}^{\mu_n}(\{x\})>0\}$ converge weakly to the closed support, $S(\hat X_t)$, of $\hat X_t$ w.r.t. the Hausdorff metric as $n\to\infty$.  Fix $t_n\to t>0$ and assume without loss of generality that $\inf_nt_n>0$.  If $\beta\in\tilde I$ and $0\le \eps\le s$, let 
\[I_n(s,\eps)=\{\gamma\in\tilde I:\gamma\approx\lfloor(s-\eps)\mu_n\rfloor,\ \exists\beta\approx\lfloor s\mu_n\rfloor\ s.t.\ \gamma<\beta\},\]
that is, $I_n(s,\eps)$ is the set of individuals in population $\tilde X^{\mu_n}_{s-\eps}$ which have descendants alive in $\tilde X^{\mu_n}_s$.  In what follows we consider $m$ large enough so that $2^{-m}\le \inf_nt_n$.  For $\gamma\in I_n(t_n,2^{-m})$, let $N_n(\gamma,t_n)=|\{\beta\approx\lfloor\mu_nt_n\rfloor:\beta>\gamma\}|$ be the number of descendants of $\gamma$ alive in the population $X_{t_n}^{\mu_n}$.  A branching process argument (with offspring law binomial $(V(R_n),p_n)$) using Lemma~2.1(c) of \cite{crabgrass}, shows that for some $C>0$, 
\[\limsup_{n\to\infty}\P(N_n(\gamma,t_n)\le 8^{-m}\mu_n|\gamma\in I_n(t_n,2^{-m}))\le 1-e^{-C4^{-m}}\le C4^{-m}.\]
Therefore 
\begin{align}
\limsup_{n\to\infty}\,&\P(N_n(\gamma,t_n)\le 8^{-m}\mu_n\ \text{for some }\gamma\in I_n(t_n,2^{-m}))\cr
&\le \limsup_{n\to\infty}\sum_{\gamma_0\le \mu_n,|\gamma|=\lfloor\mu_n(t_n-2^{-m})\rfloor}
\P(N_n(\gamma,t_n)\le 8^{-m}\mu_n|\gamma\in I_n(t_n,2^{-m}))\P(\gamma\in I_n(t_n,2^{-m}))\cr
&\le C4^{-m}\limsup_{n\to\infty}\sum_{\gamma_0\le \mu_n,|\gamma|=\lfloor\mu_n(t_n-2^{-m})\rfloor}\P(\gamma\approx\lfloor(t_n-2^{-m})\mu_n\rfloor)\P(\tilde Z_{\lfloor\mu_n2^{-m}\rfloor}(\R^d)>0|\tilde Z_0=1_{\{0\}})\cr
&\le C4^{-m}\limsup_{n\to\infty}2K_{\ref{brwsurvival}}(2\theta' 2^{-m})2^m \mu_n^{-1}\E(|\{\gamma\in \tilde I:\gamma\approx \lfloor(t_n-2^{-m})\mu_n\rfloor \}|)\cr
&\le C2^{-m}\limsup_{n\to\infty}(V(R_n)p_n)^{\lfloor(t_n-2^{-m})\mu_n\rfloor}\cr
\label{highdensity}&\le Ce^{Ct}2^{-m},
\end{align}
where $C$ may depend on $\theta'$.
In the last line we have used the growth condition on $(V(R_n),p_n)$, in the fourth inequality we have used the definition of $K_{\ref{brwsurvival}}$ from Lemma~\ref{brwsurvival} to absorb it into $C$, and in the third inequality we have used the survival probability bound in \eqref{eq:slightlysupercritical-survival} with $k=\lfloor \mu_n2^{-m}\rfloor$, for $m$ fixed. To check that the hypotheses of Lemma~\ref{brwsurvival} are in force, note that for large enough $n$,
\[1\le V(R_n)p(R_n)\le 1+\frac{2\theta'}{\mu_n}\le 1+\frac{2\theta'2^{-m}}{\lfloor\mu_n2^{-m}\rfloor}.\]
\eqref{highdensity} gives the inequality in the display just before (4.27) in
\cite{dawsoniscoeperkins} and the rest of the proof of Lemma~4.8 of \cite{dawsoniscoeperkins} now proceeds as for branching Brownian motion in 
that reference, again using our modulus of continuity in Lemma~\ref{modcont}.  The idea is 
that the above bound shows that any point in the support of $X_{t_n}^{\mu_n}$ will have enough mass nearby from its ancestor at time $t_n-2^{-m}$ for $m$ large enough that it
will be arbitrarily close to a point in the support of the limiting $\hat X_t$. \qed

\medskip
\noindent{\bf Proof of Lemma~\ref{lemma:brw-range-A/r^2}.} It clearly suffices to consider $r\in\N$.  Let $T_0=\min\{n:Z_n(\R^d)=0\}$. For $n,r\in \N$ as in the statement of the lemma we have 
\begin{equation}\label{rRrel}r^2\le cK^2R^{d-1},
\end{equation}
which implies that 
\begin{equation}\label{binomialbound}
1\le V(R)p(R)=1+\frac{\theta}{R^{d-1}}\le 1+\frac{\theta cK^2}{r^2}.
\end{equation}
Clearly we have 
\begin{align*}
\P(\calR_n\cap([-r,r]^d)^c\neq\emptyset)\le \P(T_0>r^2)+\P(\calR_{r^2}\cap([-r,r]^d)^c\neq\emptyset,\, T_0\le r^2).
\end{align*}
By the extinction bound \eqref{eq:slightlysupercritical-survival} we have for sufficiently large $r$, $\P(T_0>r^2)\le 2K_{\ref{brwsurvival}}(\theta cK^2)r^{-2}$ (if not, choose sequences $\{r_n\}$ and $\{R_n\}$ both going to $\infty$ so that we can contradict the conclusion of Lemma~\ref{brwsurvival}.) Therefore for all $r\in\N$ we have $\P(T_0>r^2)\le Br^{-2}$ for some $B=B(cK^2)$.  Hence, it suffices to show
\begin{equation*}
\P(\calR_{r^2}\cap([-r,r]^d)^c\neq\emptyset)\le Ar^{-2} \text{ for $r\in\N$ as in \eqref{rRrel} and }A=A(c,K).
\end{equation*}
Assume that this is not the case. Then there are  sequences of natural numbers $r_n\to\infty$ and $R_n\to\infty$ such that
\begin{equation}\label{rRnrel}
r_n^2\le cK^2R_n^{d-1},
\end{equation}
and 
\begin{equation}\label{rangbad}
\lim_{n\to\infty}r_n^2\P(\calR_{r_n}\cap([-r_n,r_n]^d)^c\neq\emptyset)=\infty.
\end{equation}
Recall that we have chosen $p(R_n)$ so that $V(R_n)p(R_n)=1+\theta/R_n^{d-1}$ for some $\theta>0$. This and \eqref{rRnrel} show that $V(R_n)p(R_n)\le 1+\frac{\theta K^2c}{r_n^2}$.  The probability in \eqref{rangbad} will only increase if we raise $p(R_n)$ to $p_n$ so that $V(R_n)p_n=1+\frac{\theta K^2c}{r_n^2}$ and so we may use this modified Bernoulli probability for which \eqref{rangbad} holds, and if $\mu_n=r_n^2$, then
\begin{equation}\label{pasymp}
\mu_n(V(R_n)p_n-1)=\theta K^2c=:\theta'>0.
\end{equation}
Now consider $\tilde X^{\mu_n}$ as above, and recalling \eqref{tildeZ}, we have
\begin{equation}\label{Rbound1}
\P(\tilde\calR^{\mu_n}_1\cap((-1,1)^d)^c=\emptyset)=(1-\P(\calR_{r_n}\cap((-r_n,r_n)^d)^c\neq\emptyset))^{r_n^2}\to 0\text{ as }n\to\infty,
\end{equation}
by \eqref{rangbad}. On the other hand if $B(0,1)$ is the Euclidean open unit ball, then by Lemma~\ref{BRWlimitranges} (recall \eqref{pasymp})
\begin{align}\label{Rbound2}
\liminf_{n\to\infty}\P(\tilde\calR^{\mu_n}_1\cap((-1,1)^d)^c=\emptyset)&\ge\P_{\delta_0}(\hat\calR_1\cap((-1,1)^d)^c=\emptyset)\cr
&\ge\P_{\delta_0}(\hat X_s(\overline{B(0,1)}^c)=0\ \ \forall s\ge0)\cr
&=e^{-u(0)}>0,
\end{align}
where $u$ is the unique radial solution of $\Delta u=u^2$ on $B(0,1)$ and $u(x)\to\infty$ as $|x|\uparrow 1$ (Theorem 1 of \cite{iscoe88}).  Together \eqref{Rbound1} and \eqref{Rbound2} give us the contradiction which completes the proof. \qed

\medskip
\paragraph{Acknowledgements.} The second author thanks Xinghua Zheng for a number of helpful comments. 

\bibliographystyle{plain}

%
%
%
%
%
%
%
%
%
%
\end{document}